\documentclass{article}
\usepackage{graphicx,amsfonts,amsmath,mathrsfs,amssymb,amsthm,url,color,tikz}
 \usepackage{fancyhdr,indentfirst,bm,enumerate,natbib, float,subfigure}
  \usepackage[colorlinks=true,citecolor=blue]{hyperref}
   \usepackage{dsfont}
\usetikzlibrary{arrows.meta,calc}

\def\Fbar{ {\overline F}}

  \def\oo{\infty}                   \def\d{\,\mathrm{d}}

 \newcommand{\essinf}{\mathrm{ess\mbox{-}inf}}

    \newcommand{\E}{\mathbb{E}}
     \newcommand{\R}{\mathbb{R}}
      
       \renewcommand{\P}{\mathbb{P}}
  
\renewcommand{\(}{\left (}
 \renewcommand{\)}{\right )}
  \renewcommand{\[}{\left [}
   \renewcommand{\]}{\right ]}
\newtheorem{theorem}{Theorem}[section]
 
  \newtheorem{lemma}[theorem]{Lemma}
   \newtheorem{proposition}[theorem]{Proposition}
    \newtheorem{definition}{Definition}[section]
     \newtheorem{example}[theorem]{Example}

       \newtheorem{conjecture}[theorem]{Conjecture}

\numberwithin{equation}{section}
 \numberwithin{theorem}{section}

\renewcommand{\cite}{\citet}
 \allowdisplaybreaks
  
\begin{document}

\title{Separation Theorems for Increasing Convex and Weak 3-Convex Orders on Real-Valued Random Variables}

\author{
 Yuting Su\thanks{School of Management, University of Science and Technology of China,  China.
 Email: \url{syt20020224@mail.ustc.edu.cn}}
 \and  Meng Guan\thanks{School of Management, University of Science and Technology of China, China.
Email: \url{gm123123@mail.ustc.edu.cn}}
\and Zhenfeng Zou\thanks{School of Public Policy and Management, University of Science and Technology of China,  China. Email: \url{zfzou@ustc.edu.cn}}
 \and  Taizhong Hu\thanks{School of Management, University of Science and Technology of China,  China. 
 Email: \url{thu@ustc.edu.cn}}      	}

\date{\today}

\maketitle

\begin{abstract}

We examine separation theorems for the increasing convex (stop-loss) order and weak 3-convex order. We improve prior proofs for the increasing convex order by using integrated survival functions. We extend the separation theorem for weak 3-convex order from non-negative random variables to general real-valued random variables with finite second moments via a constructive approach based on higher-order integrated survival functions. A counterexample verifies that the $3$-convex order and $2$-convex order chain does not hold for the weak $3$-convex order. Our results enrich stochastic order theories for actuarial science and risk management.
\medskip
		
\noindent \textbf{MSC2000 subject classification}: 60E15, 91B30

\noindent \textbf{Keywords}: Stochastic order; Increasing convex (stop-loss) order; Weak 3-convex order; Separation theorem; Integrated survival function
\end{abstract}

\section{Introduction}

Stochastic orders are fundamental tools in probability, actuarial science, mathematical economics and risk theory, which provide systematic ways to compare random variables according to their magnitude, variability and risk characteristics. Among various classes of stochastic order relations, convex order and increasing convex order (stop-loss order) attract extensive attention, as they are closely linked to risk aversion, risk dispersion and optimal decision-making problems. A well-known family of generalized orders, the weak $n$-convex orders and $n$-convex orders, further extend the classical framework to higher degrees, where different moment constraints and partial moment conditions are imposed to characterize higher-order risk changes.

A pivotal result for increasing convex order is the separation theorem, which decomposes an increasing convex order relation into a combination of two simpler stochastic orders. Two equivalent decompositions are widely recognized: one via the usual stochastic order followed by convex order, and the other via convex order followed by usual stochastic order. Early proofs and discussions of this theorem can be found in \cite{Mak94} and \cite{Mul96}. 

\cite{Mak94} established relevant results via the martingale representation of increasing convex orders. \cite{Mul96} discussed both decomposition forms of the separation theorem for non-negative random variables using integrated survival functions. However, when extending his method to general random variables, it only works for the first decomposition (Theorem \ref{th-260204} (1)), and cannot be applied to the second one (Theorem \ref{th-260204} (2)). Accordingly, one objective of this paper is to provide a rigorous proof for the second decomposition using integrated survival functions.

Recently, \cite{DLM14} established a separation theorem for the weak $3$-convex order, demonstrating that any weak 3-convex order shift can be decomposed into a $2$-convex order followed by a $3$-convex order. However, their construction and proof are confined to non-negative random variables only, which restricts the applicability of their results to general probabilistic models involving real-valued random variables. Furthermore, when extending the result to the weak $n$-convex order for $n\ge 4$, their proof for Proposition 4.1 relies on stationary-excess operators and Khinchine representations \citep[see, for example,][Theorem 1.3]{DJ88}. Since the invertibility and order-preserving properties of these tools are not fully justified, the argument contains notable gaps.

Motivated by the above limitations, this paper aims to generalize the separation theorem for weak $3$-convex order to the class of general real-valued random variables with finite second moments. We adopt a constructive approach using integrated survival functions, which avoids the problematic operator techniques in previous work. We first revisit the definitions and basic integral properties of higher-order integrated survival functions, as well as fundamental facts about convex and weak convex orders. Then we revisit the proofs for the separation theorem of increasing convex order.
After that, we establish the separation theorem for weak $3$-convex order over all real-valued random variables by dividing discussions into two cases based on the essential infimum of distributions. Furthermore, we construct an explicit counterexample to demonstrate that the reverse decomposition of the weak 3-convex order does not hold, which clarifies the one-sided nature of this separation property.

The rest of this paper is organized as follows. Section \ref{sect-2} collects necessary preliminaries, including the recursive definition of higher-order integrated survival functions, related integral identities, and formal definitions of $n$-convex and weak $n$-convex orders. Section \ref{sect-3} revisits and rigorously proves the separation theorem for classical increasing convex order via integrated survival functions. Section \ref{sect-4} is devoted to the main result of this paper: the separation theorem for weak $3$-convex order for real-valued random variables, together with a counterexample for the reverse direction.

Throughout, ``increasing'' and ``decreasing'' mean ``nondecreasing'' and ``nonincreasing'', respectively. All integrals and expectations are implicitly assumed to exist whenever they are written.

\section{Preliminaries}
\label{sect-2}

\subsection{Higher-degree integrated survival function}

\newpage

For $r\ge 0$, let $L^r$ be the set of all random variables on an atomless probability space $(\Omega, \mathscr{F}, \P)$ with finite $r$-th moment. For $X\in L^0$ with distribution function $F_X$, set $\Fbar_X^{[1]}(x)=\Fbar_X(x)$, and define $n$-th integrated survial function $\Fbar^{[n]}_X$ recursively by
\begin{equation}
 \label{eq-260206-1}
     \Fbar^{[n]}_X (x) = \int_x^\oo \Fbar_X^{[n-1]} (t) \d t,\quad x\in \R,\ n\ge 2.
\end{equation}
A similar argument to the proof of Proposition 1 in \cite{OR01} yields that
\begin{align}
 \label{eq-260206-2}
 \Fbar_X^{[n]}(x) & = \frac {1}{(n-1)!}\, \int^\oo_x (t-x)^{n-1} \d t 
    = \frac {1}{(n-1)!}\, \E [(X-x)_+^{n-1}],\quad  x\in\R,
\end{align}
where $a_+=\max\{0, a\}$ and $a_+^n=(a_+)^n$. From \eqref{eq-260206-2}, $\Fbar_X^{[n]}(x)$ is finite for any $x\in\R$ when $X\in L^{n-1}$ and $n\ge 1$.

Note that, for $n\ge 3$,
$$
  \Fbar_X^{[n-1]}(x) -\frac {1}{(n-2)!} \E\[(X-x)^{n-2}\] = \frac {(-1)^{n-1}}{(n-2)!} \E\[(X-x)_-^{n-2}\],
$$
which goes to zero as $x\to -\oo$. Here, $a_- =\max\{0, -a\}$ and $a_-^n=(a_-)^n$. Thus, $\E\[(X-x)^{n-2}\]/(n-2)!$ is the \emph{left asymtote} of $\Fbar^{[n-1]}_X(x)$ as $x\to -\oo$.

\begin{proposition}
  \label{pr-260201}
For $X\in L^{n-1}$ with $n\ge 2$, we have
$$
    \int_{-\oo}^0 \(\Fbar_X^{[n-1]}(u) -\frac {1}{(n-2)!} \E \[(X-u)^{n-2}\] \) \d u=\frac{(-1)^{n-1}}{(n-1)!} \E [X^{n-1}_-].
$$
\end{proposition}

\noindent \emph{Proof}.\ Since $(X-t)^{n-1}= (X-t)_+^{n-1} + (-1)^{n-1} (X-t)_-^{n-1}$ for $t\in\R$, we obtain 
\begin{align*}
   & \lim_{t\to -\oo} \int_{t}^0 \(\Fbar_X^{[n-1]}(u) -\frac {1}{(n-2)!} \E \[(X-u)^{n-2}\] \) \d u \\
   & \qquad =\lim_{t\to -\oo} \(\Fbar^{[n]}_X(t) -\Fbar^{[n]}_X(0) +\frac{1}{(n-1)!} \Big(\E [X^{n-1}] -\E [ (X-t)^{n-1}] \Big )\) \\
 & \qquad =\lim_{t\to -\oo} \frac {1}{(n-1)!} \Big \{\E [(X-t)_+^{n-1}] -\E [X_+^{n-1}] + \E [X^{n-1}] -\E [ (X-t)^{n-1}]\Big \}  \\[3pt]
   & \qquad =\lim_{t\to -\oo} \frac{(-1)^{n-1}}{(n-1)!} \Big \{\E [X^{n-1}_-] -\E [(X-t)_{-}^{n-1}]\Big\} \\[3pt]
   & \qquad = \frac{(-1)^{n-1}}{(n-1)!} \E [X^{n-1}_-].   
\end{align*}
Thus, the desired result follows.\ \qed

\begin{proposition}
  \label{pr-260202}
Let $X, Y\in L^{n-1}$ with $\E [X^k]=\E [Y^k]$ for $k\in [n-2]$. Then
$$
    \int_{-\oo}^\oo\(\Fbar^{[n-1]}_Y(x)-\Fbar^{[n-1]}_X(x)\)\d x 
    =\frac {1}{(n-1)!} \big ( \E [Y^{n-1}] -\E [X^{n-1}]\big ).
$$    
\end{proposition}

\noindent \emph{Proof}.\ By Proposition \ref{pr-260201}, it follows from \eqref{eq-260206-2} that  
\begin{align*}
 & (n-1)!\int_{-\oo}^\oo\(\Fbar^{[n-1]}_Y(x)-\Fbar^{[n-1]}_X(x)\)\d x \\[3pt]
 & \qquad =(n-1)!\int_{-\oo}^0\(\Fbar^{[n-1]}_Y(x)-\Fbar^{[n-1]}_X(x)\)\d x + \E [Y_+^{n-1}]-\E [X_+^{n-1}] \\[3pt]
  & \qquad =(n-1)!\int_{-\oo}^0\(\Fbar^{[n-1]}_Y(x)-\frac {1}{(n-2)!} \E \[(Y-u)^{n-2}\] \)\d x + \E [Y_+^{n-1}] \\[3pt]
  & \qquad \qquad +(n-1)!\int_{-\oo}^0\(\frac {1}{(n-2)!} \E\[(X-u)^{n-2}\]-\Fbar^{[n-1]}_X(x)\)\d x -\E [X_+^{n-1}]\\[3pt]
  & \qquad = (-1)^{n-1} \E [Y_-^{n-1}] + \E [Y_+^{n-1}] -(-1)^{n-1}\E [X_-^{n-1}] -\E [X_+^{n-1}] \\
  & \qquad = \E [Y^{n-1}] -\E [X^{n-1}].  
\end{align*}
This proves the proposition.\ \qed

\subsection{Higher-degree convex and weak convex orders}

We first recall from \cite{DLM14} two families of higher-order stochastic orders: $n$-convex orders and weak $n$-convex orders, which generalize classical convex-type orders to arbitrary integer degrees $n\ge 2$. 

\begin{definition}
Let $X, Y\in L^{n-1}$ for $n\ge 2$. We say that $X$ is smaller than $Y$ in the weak $n$-convex order, written $X\le_{n\hbox{-}{\rm wcx}} Y$, if $\E [X^k] = \E [Y^k]$ for $k=1, \ldots, n-2$, and 
$$
    \Fbar^{[n]}_X (x)  \le  \Fbar^{[n]}_Y (x),\quad x\in\R,
$$
or equivalently, 
$$
     \E \[(X-x)_+^{n-1}\] \le \E \[(Y-x)_+^{n-1}\],\quad x\in\R.
$$
\end{definition}

For the weak $n$-convex order, the equality of the $(n-1)$-th moment is not required, which relaxes the moment restrictions of the strict $n$-convex order.

\begin{definition}
Let $X, Y\in L^{n-1}$ for $n\ge 2$. We say that $X$ is smaller than $Y$ in the $n$-convex order, written $X\le_{n\hbox{-}{\rm cx}} Y$, if $X\le_{n\hbox{-}{\rm wcx}} Y$ and $\E [X^{n-1}] =\E [Y^{n-1}]$.
\end{definition}

In short, the $n$-convex order is a strengthened version of the weak $n$-convex order by additionally imposing the equality of the $(n-1)$-th moment. The inclusion relation below follows directly from the definitions:
$$ 
   X\le_{n\hbox{-}{\rm cx}} Y\ \Longrightarrow  X\le_{n\hbox{-}{\rm wcx}} Y.
$$

The $n$-convex and weak $n$-convex orders can be characterized via expected utility. A utility function $u$ is said to have $n$-th degree risk aversion if $(-1)^{n+1} u^{(n)}(x)\ge 0$ for all $x\in\R$.
\begin{itemize}
  \item[(1)] $X\le_{n\hbox{-}{\rm wcx}} Y$ if and only if $\E [u(X)]\le \E [u(Y)]$ for all $(n-2)$ times continuously differentiable function $u$ such that $u^{(n-2)}$ is increasing and convex.
  
  \item[(2)] $X\le_{n\hbox{-}{\rm cx}} Y$ if and only if $\E [u(X)]\le \E [u(Y)]$ for all $(n-2)$ times continuously differentiable function $u$ such that $u^{(n-2)}$ is convex.
\end{itemize}
These utility-based characterizations explain why higher convex-type orders are widely adopted to analyze portfolio choice, insurance demand and optimal risk decisions in economics and actuarial science.

For $n=2$, the $2$-convex and weak $2$-convex orders reduce to the classical convex order $\le_{\rm cx}$ and increasing convex order $\le_{\rm icx}$ (stop-loss order $\le_{\rm sl}$), respectively, which correspond to risk-averse and increasing risk-averse utility functions in standard risk theory. 

The following theorem gives well-known sample path characterizations of the orders $\le_{\rm cx}$ and $\le_{\rm icx}$. 
 
\begin{theorem}
 \label{th-260205}
{\rm \citep{Rus81,SS07}}\
\begin{itemize}
 \item[{\rm (1)}] For $X, Y\in L^1$, $X\le_{\rm cx} Y$ if and only if there exist two random variables $\widehat{X}$ and $\widehat{Y}$, defined on the same atomless probability space, such that $\widehat{X} \stackrel {d}{=} X$, $\widehat{Y} \stackrel {d}{=} Y$, and 
     $\E \big [\widehat{Y} | \widehat{X}\big ] = \widehat{X},\ {\rm a.s.}.$
 
 \item[{\rm (2)}] For $X, Y\in L^1$, $X\le_{\rm icx} Y$ if and only if there exist two random variables $\widehat{X}$ and $\widehat{Y}$, defined on the same atomless probability space, such that $\widehat{X} \stackrel {d}{=} X$, $\widehat{Y} \stackrel {d}{=} Y$, and 
     $\E \big [\widehat{Y} | \widehat{X}\big ] \ge  \widehat{X},\ {\rm a.s.}.$
\end{itemize}
\end{theorem}

Theorem \ref{th-260205} was entended by \cite{MR01}, in which the two step submartingale $(\widehat{X}, \widehat{Y})$ has the additional property that the conditional distribution of $ \big [\widehat{Y} | \widehat{X}=x\big ]$ is stochastically increasing in $x$. 

In the above definition of the convex order, we always assume that the random variables have finite means. For the definition of the convex order between two random variables with infinite mean, see \cite{CW26}. For a new characterization of the increasing convex order, see \cite{GHW24}.

\section{Separation theorem for the usual increasing convex order}
\label{sect-3}

For notational convenience, in this section we denote $\Psi_X =\Fbar_X^{[2]}$ for $X\in L^1$, which is also called the stop-loss transform of $X$ in actuarial science. We first recall some basic properties of $\Psi_X$ from \cite{MR01}, which are stated as the following two lemmas.

\begin{lemma}
  \label{le-260200} 
\begin{itemize}
 \item[{\rm (1)}] For $X\in L^1$, the integrated survival function $\Psi_X$ has the following properties: {\rm (i)}\ $\Psi_X$ is decreasing and convex; {\rm (ii)}\ the right derivative $D^+\Psi_X$ exists and $-1\le  D^+\Psi_X\le 0$; {\rm (iii)}\ $\lim_{x\to \oo} \Psi_X(x)=0$ and $\lim_{x\to -\oo} [x+\Psi_X(x)]=\E X$.

 \item[{\rm (2)}] For every function $\Psi:\R\to \R$ which fulfills {\rm (i), (ii)} and {\rm (iii)} in part {\rm (1)}, there is a random variable $X\in L^1$ such that $\Psi$ is the integrated survival function of $X$. The distribution function of $X$ is given by $F_X(x)=D^+\Psi(x)+1$ for $x\in \R$.
\end{itemize}
\end{lemma}

\begin{lemma}
  \label{le-260203} 
\begin{itemize}
 \item[{\rm (1)}] $X\le_{\rm st} Y$ if and only if $\Psi_Y(x)-\Psi_X(x)$ is decreasing in $x\in\R$.

 \item[{\rm (2)}] For $X, Y\in L^1$, $X\le_{\rm cx} Y$ if and only if $\Psi_X(x)\le \Psi_Y(x)$ for all $x\in\R$ and 
     $$ \lim_{x\to -\oo} \left [\Psi_Y(x)-\Psi_X(x)\right ]=0.    $$
\end{itemize}
\end{lemma}

The classical separation theorem for the increasing convex order $\le_{\rm icx}$ is recalled below, where $\le_{\rm icx}$ is characterized by the interplay between the usual stochastic order $\le_{\rm st}$ and the convex order $\le_{\rm cx}$.

\begin{theorem}
 \label{th-260204}
{\rm \citep[][Theorem 4.A.6]{SS07}}\ \ Let $X$ and $Y$ be two random variables.
\begin{itemize}
  \item[{\rm (1)}] $X\le_{\rm icx} Y$ if and only if there exists a random variable $Z$ such that $X\le_{\rm st} Z\le_{\rm cx} Y$.
  \item[{\rm (2)}] $X\le_{\rm icx} Y$ if and only if there exists a random variable $Z$ such that $X\le_{\rm cx} Z\le_{\rm st} Y$.
\end{itemize}
\end{theorem}

\cite{Mak94} proved Theorem \ref{th-260204} by a sample path characterization of the order $\le_{\rm icx}$ (Theorem \ref{th-260205} (2)). An alternative proof of Theorem \ref{th-260204} can also be found in \cite{Mak94}, who first gave a constructive proof for distributions with finite support and rational probabilities, and then extended the result to the class of general distribution functions on $\R$ by a multistep approximation procedure via limiting arguments.

\cite{Mul96} proved Theorem \ref{th-260204} for nonnegative random variables by exploiting Lemmas \ref{le-260200} and \ref{le-260203}, and his method still applies to the first part of Theorem \ref{th-260204} for general random variables, but not to the second part. 

We now give an alternative proof of the second part of Theorem \ref{th-260204} based on the integrated survial function, and the random variable $Z$ is constructed directly therein.
\medskip

\noindent \underline{\emph{Proof of Theorem \ref{th-260204}(2)}}:\ \ It suffices to prove the necessity. We aim to explicitly construct such a random variable $Z$. If $\E X = \E Y$, we simply take $Z=Y$. Now consider the nontrivial case $\E X <\E Y$. Since $X\le_{\rm icx}Y$, we have $\Psi_X(x)\le \Psi_Y(x)$ for all $x\in \R$. The following discussion will be divided into two cases based on whether the essential infimum of $Y$ is finite.

\emph{Case 1}.\ \ Suppose $\essinf(Y)= x_0\in\R$. Then $\Psi_Y(x)=\E Y-x$ for $x\le x_0$. Two subcases arise. First, suppose $\Psi_X(x_0)< \Psi_Y(x_0)$. We use an idea from the proof of Theorem 3.7(b) in \cite{Mul96}. There exists some $\alpha\in [-1,0]$ such that the linear function $x\mapsto \Psi_X(x_0)+\alpha (x-x_0)$ is a tangent for the graph of $\Psi_Y(x)$, that is, there exists some $x_1>x_0$ such that $\Psi_Y(x_1)=\Psi_X(x_0) +\alpha (x_1-x_0)$, and $\Psi_Y(x) \ge \Psi_X(x_0)+\alpha (x-x_0)$ for $x>x_0$. Here, $\alpha$ can be chosen as the left derivative $D^{-}\Psi_Y(x_1)= \Fbar_Y(x_1-)$. Define the function
$$
   \Psi_Z(x) =\left\{\begin{array}{ll} \Psi_X(x), & x< x_0,\\
     \Psi_X(x_0) +\alpha (x-x_0), & x_0\le x< x_1,\\
     \Psi_Y(x), & x \ge x_1.    \end{array}\right.
$$
It is easy to see that $\Psi_Z(x)$ is a integrated survival function of some random variable $Z$, and that $\Psi_Z(x)\ge \Psi_X(x)$ for $x\in\R$; see Figure \ref{fig-1}. Thus, by Lemma \ref{le-260203}, $X\le_{\rm cx} Z$. On the other hand, note that
$$
   \Psi_Y(x)-\Psi_Z(x) =\left\{\begin{array}{ll} \E Y-x -\Psi_X(x), & x\le x_0,\\
               \Psi_Y(x)-\Psi_X(x_0)-\alpha (x-x_0), & x_0<x \le x_1,\\
               0, & x>x_1,  \end{array} \right.
$$
which is decreasing in $x\in\R$. Hence, again by Lemma \ref{le-260203}, $Z\le_{\rm st} Y$.

\begin{figure}[hbt]
\begin{center}
\begin{tikzpicture}\tikzset{
    >=Stealth,
    dot/.style={circle,fill,black,minimum size=4pt,inner sep=0pt},
    dashed line/.style={dashed,draw=black!60},
    main line/.style={thin}
    }

\draw[->] (-4.5,0) -- (4.5,0) node[right] {$x$};
\draw[->] (0,-0.5) -- (0,3.5) node[above] {$y$};

\def\psiY{
    \x < -2.5 ? -\x + 0.4*exp(0.5*2.5) - 2.5 : 0.4*exp(-0.5*\x)
}
\def\psiX{
    1.0*exp(-0.5*(\x + 3))
}

\draw[main line,black] plot[variable=\x,domain=-4.0:4.5,samples=200] (\x,{\psiY});
\draw[main line,red] plot[variable=\x,domain=-4.0:4.5,samples=200] (\x,{\psiX});

\coordinate (x0) at (-2.5, 0.7788);
\coordinate (x1) at (0.5, 0.31152);

\coordinate (Zleft)  at (-3.8, 0.981288);
\coordinate (Zright) at (2.5, 0);

\draw[main line,blue,dashed] (Zleft) -- (x0);
\draw[main line,blue]        (x0) -- (x1);
\draw[main line,blue,dashed] (x1) -- (Zright);

\node[dot] at (x0) {};
\node[dot] at (x1) {};

\coordinate (Yx0) at (-2.5, 1.396137);
\draw[dashed line] (Yx0) -- (-2.5,0) node[below] {$x_0$};
\draw[dashed line] (x1) -- (0.5,0) node[below] {$x_1$};

\node[black,above] at (-3, 2.1) {$\Psi_Y$};
\node[red,below] at (-3.2, 1.7) {$\Psi_X$};
\node[blue,align=center] at (-2, 0.9) {$\Psi_Z$};
\end{tikzpicture}
\end{center}
\caption{The plots of $\Psi_X$, $\Psi_Y$ and $\Psi_Z$.}
 \label{fig-1}
\end{figure}
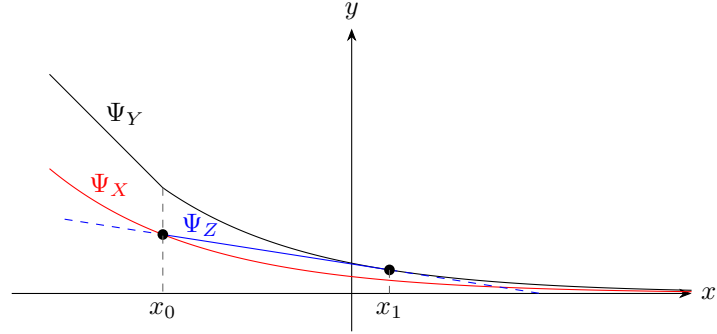

Next, suppose suppose $\Psi_X(x_0)= \Psi_Y(x_0)$. Define
$$ 
   \Psi_Z(x)=\left\{ \begin{array}{ll}    \Psi_X(x), & x<x_0,\\
            \Psi_Y(x), & x\ge x_0. \end{array} \right.
$$
Since $\Psi_X(x)\le \Psi_Y(x)$ for all $x\in \R$ and $\Psi_X(x_0)=\Psi_Y(x_0)$, it follows that $D^-\Psi_X(x_0)\le D^+\Psi_X(x_0)\le D^+\Psi_Y(x_0)$, which implies $\Psi_Z$ is convex. Furthermore, one can verify $\Psi_Z$ is an integrated survival function for some random variable $Z$. Besides, it is straightforward to confirm that $\Psi_Z(x)\ge \Psi_X(x)$ for all $x\in \R$, and $\Psi_Y(x)-\Psi_Z(x)$ is decreasing in $x\in\R$. Consequently, by Lemma \ref{le-260203}, we obtain $X\le_{\rm cx} Z\le_{\rm st}Y$. This proves Case 1.

\emph{Case 2}.\ \ Suppose $\essinf(Y)= -\oo$. Then $D^+\Psi_Y(x)=-\Fbar_Y(x)\in (-1,0]$ for all $x\in\R$. Define 
$$
    \varphi(x) =\Psi_X(x)+\int_{-\oo}^x \(\Fbar_X(t)-\Fbar_Y(t)\)_+ \d t,\quad x\in\R,
$$
which satisfies $D^+\varphi(x)=- \min\left\{\Fbar_X(x), \Fbar_Y(x)\right \}\in (-1,0]$. Thus $\varphi$ is decreasing and convex. Since 
$$
    \E Y -\E X =\lim_{x\to -\oo} [\Psi_Y(x)-\Psi_X(x)] =\lim_{x\to -\oo} [\Psi_Y(x)-\varphi(x)] >0,
$$ 
there exists $y_1\in\R$ such that $\Psi_Y(x)>\varphi(x)$ for all $x\in(-\oo, y_1]$. Denote 
$$
     \ell(x) = \Psi_Y(y_1) -\Fbar_Y(y_1) (x-y_1),\quad x\in\R. 
$$
Since $\ell(y_1)>\varphi(y_1)$ and $-\Fbar_Y(y_1)>-1=\lim_{x\to -\oo} D^+\varphi(x)$, there exists a point $y_0< y_1$ such that $\ell(x)>\varphi(x)$ for $x\in (y_0, y_1]$ and $\ell(x)<\varphi(x)$ for $x\in (-\oo, y_0)$ by the convexity of $\varphi$. Next, we define 
$$
    \Psi_Z(x)=\left\{\begin{array}{ll}   \varphi(x), & x < y_0,\\
                 \ell(x), & y_0\le x< y_1,\\
                 \Psi_Y(x), & x\ge y_1;  \end{array} \right.
$$
see Figure \ref{fig-2}. One can verify that $\Psi_Z$ is an integrated survival function of some random variable $Z$. Morever, when $x\in (-\oo, y_1)$, we have $\Psi_Z(x)= \max\{\varphi(x), \ell(x)\} \ge \varphi(x)\ge \Psi_X(x)$. Thus, $\Psi_Z(x)\geq \Psi_X(x)$ holds for all $x\in \R$. Also, $\E Z- \E X =\lim_{x\to -\oo} [\varphi(x) -\Psi_X(x)]=0$. Hence, $X\le_{\rm cx}Z$.

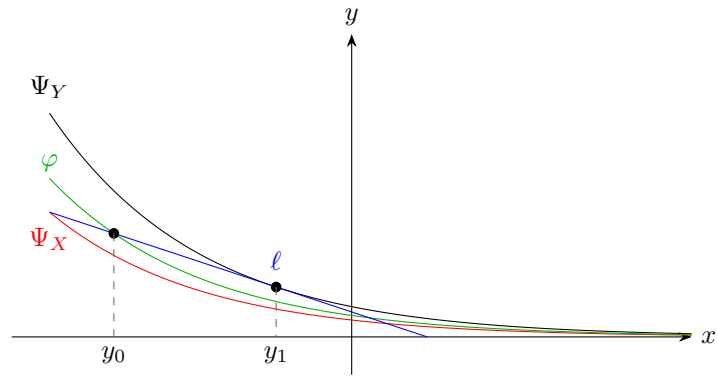
\begin{figure}[hbt]
\begin{center}
\begin{tikzpicture}[
    >=Stealth,
    dot/.style={circle,fill,black,minimum size=4pt,inner sep=0pt},
    dashed line/.style={dashed,draw=black!60},
    main line/.style={thin}
]
\draw[->] (-4.5,0) -- (4.5,0) node[right] {$x$};
\draw[->] (0,-0.5) -- (0,4) node[above] {$y$};
\def\yone{-1}
\def\yzero{-3.148054}
\draw[main line,black] 
    plot[variable=\x,domain=-4.0:4.5,samples=200] 
    (\x, {0.4 * e^(-0.5*\x)});
\draw[main line,red] 
    plot[variable=\x,domain=-4.0:4.5,samples=200] 
    (\x, {1.0 * e^(-0.5*(\x+3))});
\draw[main line,green!70!black] 
    plot[variable=\x,domain=-4.0:4.5,samples=200] 
    (\x, {0.6 * e^(-0.5*(\x+1.5))});
\draw[main line,blue] 
    plot[variable=\x,domain=-4.0:1.0,samples=100] 
    (\x, {-0.329744*\x + 0.329744});
\node[black]    at (-4.0, 3.3) {$\Psi_Y$};
\node[red]      at (-4.0, 1.3) {$\Psi_X$};
\node[green!70!black] at (-4.0, 2.3) {$\varphi$};
\node[blue]     at (-1.0, 1) {$\ell$};
\coordinate (P1) at (\yone, {0.4 * e^(-0.5*\yone)});
\coordinate (P0) at (\yzero, {0.6 * e^(-0.5*(\yzero+1.5))});
\node[dot] at (P1) {};
\node[dot] at (P0) {};
\draw[dashed line] (P1) -- (\yone,0) node[below] {$y_1$};
\draw[dashed line] (P0) -- (\yzero,0) node[below] {$y_0$};
\end{tikzpicture}
\end{center}
\caption{The plots of $\Psi_X$, $\Psi_Y$ and $\varphi$.}
 \label{fig-2}
\end{figure}

On the other hand, we can compute that
$$
   D^+\(\Psi_Y(x)-\Psi_Z(x)\) =\left\{ \begin{array}{ll}  -\(\Fbar_Y(x)-\Fbar_X(x)\)_{+}, & x < y_0,\\
                \Fbar_Y(y_1)-\Fbar_Y(x), &  y_0\le x< y_1,\\
                   0, & x\ge y_1. \end{array} \right.
$$
Thus, $D^+\(\Psi_Y(x)-\Psi_Z(x)\)\le 0$ for all $x\in\R$, impling that $\Psi_Y(x)-\Psi_Z(x)$ is decreasing in $x\in\R$. By Lemma \ref{le-260203}, we have $Z\le_{\rm st} Y$. This completes the proof. \ \qed
\medskip

\cite{MR01} proved Theorem \ref{th-260205} (2) via Theorems \ref{th-260204} (1) and \ref{th-260205} (1). Their proof, however, is not rigorous. More precisely, for three random variables $\widehat{X}, \widehat{Y}$ and $\widehat{Z}$ defined on the same probability space, if $\widehat{Z}\ge \widehat{X}$, $\E \big [\widehat{Y} | \widehat{Z}\big ]= \widehat{Z}$, a.s., and both $\big [ \widehat{Z}| \widehat{X}=z\big ]$ and $\big [ \widehat{Y}| \widehat{Z}=z\big ]$ are stochastically increasing in $z$, we cannot conclude that $ \E \big [\widehat{Y} | \widehat{X}\big ]\ge \widehat{X}$, a.s.. This is mainly because, in this case, the $\sigma$-fields generated by $\widehat{Z}$ and $\widehat{X}$ may not be nested.

Now, we prove Theorem \ref{th-260205} (2) via Theorems \ref{th-260204} (2) and  \ref{th-260205} (1).
\medskip

\noindent \underline{\emph{Proof of Theorem \ref{th-260205} (2)}}:\ \ It suffices to prove the necessity. Assume $X\le_{\rm icx} Y$. By Theorem \ref{th-260204} (2), there exists a random variable $Z$ such that $X\le_{\rm cx} Z\le_{\rm st} Y$.
By Theorem \ref{th-260205} (1), we can construct two random variables $\widehat{X}$ and $\widehat{Z}$ on the same atomeless probability space $\big (\widehat{\Omega}, \widehat{\mathscr{F}}, \widehat{\P}\big )$ such that $\widehat{X} \overset{d}{=} X$, $\widehat{Z}\overset{d}{=} Z$ and 
$$
    \E \[\widehat{Z} \big |\widehat{X}\] = \widehat{X},\, {\rm a.s.}.
$$
Since the probability space is atomless, there exists a standard uniformly $(0,1)$-distributed random variable $U_Z$ defined on  $\big (\widehat{\Omega}, \widehat{\mathscr{F}}, \widehat{\P}\big )$ such that $\widehat{Z}=F_Z^{-1}(U_Z)$, a.s.. The existence of $U_Z$ was given in Lemma A.28 of \cite{FS11}; also see \cite{Del12}. Define $\widehat{Y}=F_Y^{-1}(U_Z)$. Then $Z\le_{\rm st} Y$ implies $\widehat{Y}\ge \widehat{Z}$, a.s.. Therefore,
$$
  \E\[\widehat{Y} \big |\widehat{X}\] =\E\[\widehat{Z}+\(\widehat{Y}-\widehat{Z}\) \big | \widehat{X}\] = \widehat{X} + \E\[\widehat{Y}-\widehat{Z} \big|\widehat{X}\] \ge \widehat{X}, \, {\rm a.s.}.
$$
This completes the proof of the theorem.\ \qed

\section{Separation theorem for the weak 3-convex order}
\label{sect-4}

The weak $3$-convex order $\le_{\rm 3\hbox{-}wcx}$ admits a separation via $\le_{\rm 2\hbox{-}cx}$ and $\le_{\rm 3\hbox{-}cx}$, as stated in the following theorem. \cite{DLM14} proved this in their Proposition 3.3 for nonnegative random variables. Their method is constructive, and does not apply to general random variables.

\begin{theorem}
 \label{th-260207}
Given two random variables $X,Y\in L^2$, $X\le_{\rm 3\hbox{-}wcx} Y$ if and only if there exists a random variable $Z\in L^2$ such that $X\le_{\rm 2\hbox{-}cx} Z\le_{\rm 3\hbox{-}cx} Y$.
\end{theorem}

\begin{proof}
It suffices to prove the necessity. Suppose $X\le_{\rm 3\hbox{-}wcx}Y$. Without loss of generality, we assume $\E X =\E Y=0$ and $\E [X^2]< \E [Y^2]$.  By Proposition \ref{pr-260201} with $n=3$, we have 
$$
   \int_{-\oo}^0 \(\Fbar^{[2]}(u) +u\) \d u   =\frac{1}{2} \E [X_{-}^2].
$$
Together with the fact that $\Fbar_X^{[2]}(u) + u=\Fbar_X^{[2]}(u)-(\E X-u) \ge 0$ for $u\in\R$, we have 
$$
     \lim_{t\to -\oo} \int_{-\oo}^t  \left |\Fbar_X^{[2]}(u) + u\right | \d u=0. 
$$
Similarly, we have 
$$
    \lim_{t\to -\oo} \int_{-\infty}^t \left |\Fbar^{[2]}_Y(u) +u\right | \d u=0. 
$$
Consequently, 
\begin{equation}
 \label{eq-260207}
   \lim_{t\to -\oo} \int_{-\oo}^t \left |\Fbar^{[2]}_Y(u)-\Fbar^{[2]}_X(u)\right | \d u=0.
\end{equation}

We aim to explicitly construct such a random variable $Z$ satisfying $X\le_{\rm 2\hbox{-}cx} Z\le_{\rm 3\hbox{-}cx} Y$. To this end, we examine the following two cases, depending on whether $x_0$ is finite or infinite, where $x_0=\inf\big \{x\in\R: \Fbar^{[2]}_X(x)\ge \Fbar_Y^{[2]}(x)\big \}$.

\emph{Case 1}.\ Suppose $x_0= -\oo$. In view of \eqref{eq-260207}, we choose $x_1\in \R$ such that $\Fbar_X^{[2]}(x_1)\ge\Fbar_Y^{[2]} (x_1)$ and $\int_{-\oo}^{x_1} \big |\Fbar_Y^{[2]}(u)-\Fbar_X^{[2]}(u)\big | \d u <(\E [Y^2]- \E [X^2])/2$. For $c>x_1$, we define 
$$
   \Psi_c(x)=\left\{ \begin{array}{ll}
    \max\left\{\Fbar_Y^{[2]}(x),\Fbar^{[2]}_X(x)\right\}, & x\le x_1,\\[3pt]
    \dfrac{\Fbar_X^{[2]}(c)-\Fbar_X^{[2]}(x_1)}{c-x_1}(x-x_1) +\Fbar_X^{[2]} (x_1), & x_1<x<c,\\[3pt]
           \Fbar_X^{[2]}(x), & x>c,   \end{array}\right.
$$
where the parameter $c$ is to be determined. A straightforward calculation gives that
\begin{align*}
   \lim_{c\to x_1}\int_{-\oo}^\oo \(\Psi_c(x)-\Fbar^{[2]}_X(x)\) \d x 
   & =\int_{-\oo}^{x_1} \(\max\left\{\Fbar^{[2]}_Y(x),\Fbar_X^{[2]}(x)\right\} -\Fbar_X^{[2]}(x)\) \d x\\
   & \le \int_{-\oo}^{x_1} \left |\Fbar_Y^{[2]}(x)- \Fbar_X^{[2]}(x)\right | \d x < \frac {1}{2} \(E [Y^2]-\E [X^2]\),
\end{align*}
and
$$
   \lim_{c\to \oo} \int_{-\oo}^\oo \(\Psi_c(x)-\Fbar_X^{[2]} (x)\) \d x  = + \infty.
$$
Thus, there exists $c_1>x_1$ such that 
$$
   \int_{-\oo}^\oo \(\Psi_{c_1}(x)-\Fbar_X^{[2]}(x) \)\d x = \frac {1}{2} \(\E [Y^2] -\E [X^2]\)
$$ 
by the continuity of $\int_{-\oo}^\oo \big (\Psi_c(x)-\Fbar_X^{[2]}(x)\big )\d x$ with respect to $c$. One can verify that $\Psi_{c_1}=\Fbar^{[2]}_Z$ for some random variable $Z$; see Figure \ref{fig-3}. Observe that 
$$
   \lim_{t\to -\oo} \(\Fbar_Z^{[2]}(t)+t\) = \lim_{t\to -\oo}\(\max\left\{\Fbar_Y^{[2]}(t),\Fbar_X^{[2]}(t)\right \} +t\)=\max\{\E Y, \E X\}=0,
$$
which implies $\E Z=0$. Morever, the convexity of $\Fbar_X^{[2]}$ yields $\Fbar_Z^{[2]}(x)\ge \Fbar_X^{[2]}(x)$ for all $x\in \R$. Hence, $X\le_{\rm 2\hbox{-}cx}Z$. On the other hand, 
$$
  \frac {1}{2}\(\E [Z^2]-\E [X^2]\)=\int_{-\oo}^\oo \(\Fbar_Z^{[2]}(x)-\Fbar_X^{[2]}(x)\)\d x =\frac {1}{2}\(\E [Y^2]-\E [X^2]\),
$$
implying $\E [Z^2] =\E [Y^2]$. Next, we define
\begin{equation}
 \label{eq-260210-1}
    \eta(x) := \Fbar_Z^{[3]}(x)-\Fbar_Y^{[3]}(x),\quad  x\in\R. 
\end{equation}
Observe that $\eta'(x)=-\big (\Fbar_X^{[2]}(x)-\Fbar_Y^{[2]}(x)\big )_+\le 0$ for $x<x_1$, so $\eta(x)$ is decreasing on $(-\oo,x_1)$. Since $\lim_{x\to -\oo} \eta(x) =(\E [Z^2]-\E [Y^2])/2=0$, it follows that $\eta(x)\le 0$ for all $x\le x_1$. Additionally, for $x\ge c_1$,
$$
   \eta(x)=\int_x^\oo \(\Fbar_X^{[2]}(t)-\Fbar_Y^{[2]}(t)\)\d t=\Fbar_X^{[3]}(x)-\Fbar_Y^{[3]}(x)\le 0
$$
since $X\le_{3\hbox{-}{\rm wcx}} Y$. Morever, $\Fbar_Z^{[2]}(x)$ is linear on the interval $(x_1, c_1)$ and satisties $\Fbar^{[2]}_Z(x_1)\ge \Fbar_Y^{[2]}(x_1)$. By the convexity of $\Fbar_Y^{[2]}$, $\Fbar_Z^{[2]}$ either down-crosses $\Fbar_Y^{[2]}$ exactly once or dominates $\Fbar_Y^{[2]}$ throughout $(x_1,c_1)$. Equivalently, $\eta(x)$ is either decreasing on $(x_1,c_1)$ or first decreasing then increasing on $(x_1,c_1)$. In either case $\eta(x)\le \max\{\eta(x_1),\eta(c_1)\}\le 0$ for all $x\in(x_1,c_1)$. Therefore, $\eta(x)\le 0$ for all $x\in\R$, and thus $Z\le_{3\hbox{-}{\rm cx}} Y$. This proves Case 1.

\begin{figure}[hbt]
\begin{center}
\begin{tikzpicture}
\tikzset{
		box/.style ={
		circle, 
		minimum width =1.5pt, 
		minimum height =1.5pt, 
		inner sep=1.5pt, 
		draw=black, 
		fill=white
		}
		}
\draw[->,thick] (-5,0) -- (4,0) node[right] {$x$};
\draw[->,thick] (-1,0) -- (-1,4.8) node[above] {$y$};			
\draw[domain=-4:0] plot (\x,{0.2*\x*(\x-1)+0.5});
\draw[domain=0:3] plot (\x,{0.2*exp(-\x)+0.3})node[above] {$\Fbar_Y^{[2]}$};
\draw[domain=-4:-2,color=red] plot (\x,{-\x});
\draw[domain=-2:2,color=red] plot (\x,{-0.5*\x+1})node[above] {$\Fbar_X^{[2]}$};
\draw[domain=-2.4:0,color=blue] plot (\x,{-0.6*\x+1})node[above] {$\Fbar_Z^{[2]}$};
\node at (-2.4,2.4) { $\bullet$};
\node at (0,1) { $\bullet$};
\draw  [dash pattern={on 1.5pt off 2pt}] (0,0) -- (0,1);
\draw  [dash pattern={on 1.5pt off 2pt}] (-2.4,0) -- (-2.4,2.4);
\node at (0,-0.3) {$c_1$};
\node at (-2.4,-0.3) {$x_1$};
\node at (-4.8,-0.3) {$x_0$};
\end{tikzpicture} 
\end{center}
\caption{The plots of $\Fbar_X^{[2]}$, $\Fbar_Y^{[2]}$ and $\Fbar_Z^{[2]}$.}
 \label{fig-3}
\end{figure}
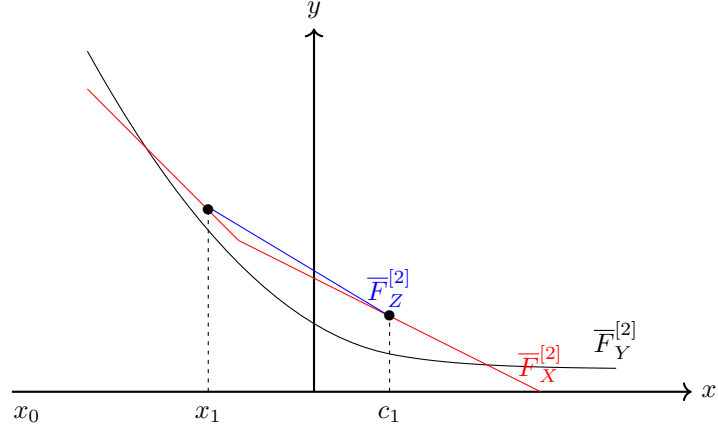

\emph{Case 2}.\ Suppose $x_0\in \R$. Choose $x_2<x_0$. For each $\alpha\in \[-\Fbar_Y(x_2),0\)$, we define
$$
   \Psi_{\alpha,c_2} (x)=\left\{\begin{array}{ll}
     \Fbar_Y^{[2]}(x), & x\le x_2,\\
    \alpha (x-x_2)+\Fbar_Y^{[2]}(x_2), & x_2<x<c_2,\\[3pt]
    \Fbar_X^{[2]} (x), & x>c_2,          \end{array} \right.
$$
where $c_2$ is choosen such that $\Psi_{\alpha,c_2}$ is continuous. When $\alpha =-\Fbar_Y (x_2)$, the convexity of $\Fbar_Y^{[2]}$ implies $\Psi_{\alpha,c_2}(x)\le \Fbar_Y^{[2]}(x)$ for all $x\le c_2$. Hence
\begin{align*}
   \int_{-\oo}^\oo \[\Psi_{\alpha,c_2}(x)-\Fbar_Y^{[2]}(x) \] \d x 
   & =\int_{-\oo}^{c_2}\[\Psi_{\alpha,c_2}(x)-\Fbar_Y^{[2]}(x)\]\d x
           +\int_{c_2}^{\infty} \[\Fbar_X^{[2]}(x)- \Fbar_Y^{[2]} (x)\] \d x\\
   & \le\int_{c_2}^\oo \[\Fbar_X^{[2]}(x)- \Fbar_Y^{[2]}(x)\]\d x =\Fbar_X^{[3]}(c_2)-\Fbar_Y^{[3]}(c_2)\le 0.
\end{align*}
Morever, 
$$
    \lim_{\alpha\to 0}\int_{-\oo}^\oo \(\Psi_{\alpha,c_2}(x)-\Fbar_Y^{[2]} (x)\) \d x= +\oo. 
$$
Consequently, there exists $\alpha_0\in \[-\Fbar_Y(x_2),0\)$ such that $\int_{-\oo}^\oo \big (\Psi_{\alpha_0,c_2}(x) -\Fbar_Y^{[2]}(x)\big ) \d x=0$. It is easy to see that $\Psi_{\alpha_0,c_2}= \Fbar_Z^{[2]}$ for some random variable $Z$; see Figure \ref{fig-4}. Observe that 
$$
    \lim_{t\to -\oo} \(\Fbar_Z^{[2]}(t)+t\) =\lim_{t\to -\oo} \(\Fbar_Y^{[2]}(t)+t\)= \E Y=0,
$$
which implies that $\E Z=0$. Additionly, by the convexity of $\Fbar_X^{[2]}$, we have $\Fbar_Z^{[2]}(x)\ge \Fbar_X^{[2]} (x)$ for all $x\in \R$. Hence, $X\le_{2\hbox{-}{\rm cx}} Z$. On the other hand, by Proposition \ref{pr-260202} with $n=3$,
$$
   \E [Z^2] -\E [Y^2] = 2 \int_{-\oo}^\oo \(\Fbar_Z^{[2]} (x) -\Fbar_Y^{[2]}(x)\) \d x=0.
$$
Following the same line of reasoning as in Case 1, let $\eta(x)$ be defined by \eqref{eq-260210-1}. Observing that $\eta'(x)= 0$ for $x<x_2$, and $\lim_{x\to -\oo} \eta(x) =0$, we have $\eta(x)= 0$ for all $x\le x_2$. Additionally, for $x\ge c_2$,
$$
   \eta(x)=\int_x^\oo \(\Fbar_X^{[2]}(t)-\Fbar_Y^{[2]} (t)\)\d t =\Fbar_X^{[3]}(x)-\Fbar_Y^{[3]}(x)\le 0
$$
since $X\le_{3\hbox{-}{\rm wcx}} Y$. Morever, $\Fbar_Z^{[2]}(x)$ is linear on the interval $(x_2, c_2)$ and satisties $\Fbar_Z^{[2]}(x_2)\ge \Fbar_Y^{[2]}(x_2)$. By the convexity of $\Fbar_Y^{[2]}$, $\Fbar_Z^{[2]}$ either down-crosses $\Fbar_Y^{[2]}$ exactly once or dominates $\Fbar_Y^{[2]}$ throughout $(x_2,c_2)$. Equivalently, $\eta(x)$ is either decreasing on $(x_2,c_2)$ or first decreasing then increasing on $(x_2,c_2)$. In either case $\eta(x)\le \max\{\eta(x_2),\eta(c_2)\}\le 0$ for all $x\in (x_2,c_2)$. Therefore, $\eta(x)\le 0$ for all $x\in \R$, and thus $Z\le_{3\hbox{-}{\rm cx}} Y$. This completes the proof of the theorem.
\end{proof}

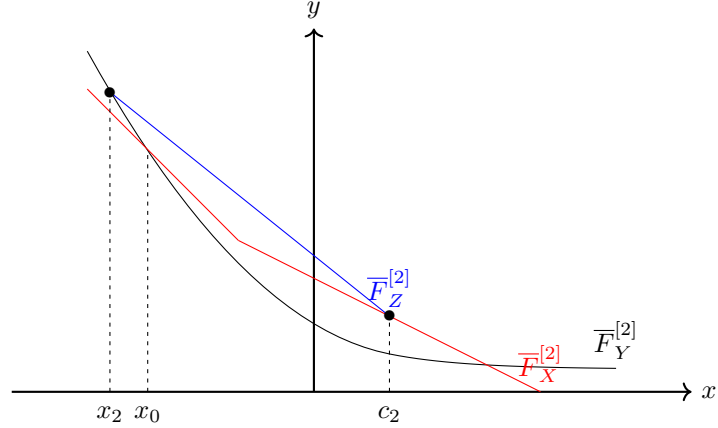
\begin{figure}[htb]
\begin{center}
\begin{tikzpicture}
\tikzset{
		box/.style ={
		circle, 
		minimum width =1.5pt, 
		minimum height =1.5pt, 
		inner sep=1.5pt, 
		draw=black, 
		fill=white
		}
		}
\draw[->,thick] (-5,0) -- (4,0) node[right] {$x$};
\draw[->,thick] (-1,0) -- (-1,4.8) node[above] {$y$};			
\draw[domain=-4:0] plot (\x,{0.2*\x*(\x-1)+0.5});
\draw[domain=0:3] plot (\x,{0.2*exp(-\x)+0.3})node[above] {$\Fbar_Y^{[2]}$};
\draw[domain=-4:-2,color=red] plot (\x,{-\x});
\draw[domain=-2:2,color=red] plot (\x,{-0.5*\x+1})node[above] {$\Fbar_X^{[2]}$};
\draw[domain=-3.7:0,color=blue] plot (\x,{-0.8*\x+1})node[above] {$\Fbar_Z^{[2]}$};
\node at (-3.7,3.95) { $\bullet$};
\node at (0,1) { $\bullet$};
\draw  [dash pattern={on 1.5pt off 2pt}] (0,0) -- (0,1);
\draw  [dash pattern={on 1.5pt off 2pt}] (-3.7,0) -- (-3.7,3.95);
\draw  [dash pattern={on 1.5pt off 2pt}] (-3.2,0) -- (-3.2,3.2);
\node at (0,-0.3) {$c_2$};
\node at (-3.7,-0.3) {$x_2$};
\node at (-3.2,-0.3) {$x_0$};
\end{tikzpicture} 
\end{center}
\caption{The plots of $\Fbar_X^{[2]}$, $\Fbar_Y^{[2]}$ and $\Fbar_Z^{[2]}$. }
 \label{fig-4}
\end{figure}

Theorem \ref{th-260207} states that any weak $3$-convex order shift can be decomposed into a first shift in the $2$-convex order followed by a second shift in the $3$-convex order. Inspired by Theorem \ref{th-260204}, one naturally wonders whether there exists another version of the weak $3$-convex order separation result for Theorem  \ref{th-260207}, namely that any weak $3$-convex order shift can be decomposed into a $3$-convex order shift followed by a $2$-convex order shift. The answer to this question is negative, as can be demonstrated by the following counterexample.

\begin{example}
 {\rm
Let $X$ be a discrete random variable with probability mass function given by 
$$
   \P\(X=-\frac {4}{4e^2-1}\) =\frac {4e^2-1}{4e^2} = 1-\P(X=4),
$$
and $Y$ follow the standard Laplace distribution, i.e., $Y$ has the density function $f_Y(y)=e^{-|y|}/2$ for $y\in\R$. A straightforward computation yields that
\begin{align*}
  \Fbar_X^{[2]}(t) & =\E [(X-t)_+] =\left\{\begin{array}{ll}
     -t, & t\le -\dfrac {4}{4e^2-1},\\[5pt]
     \dfrac {1}{4e^2} (4-t), & -\dfrac {4}{4e^2-1}<t<4,\\[3pt]
     0, & t\ge 4; \end{array}\right.  \\[5pt]
  \Fbar_X^{[3]}(t) & =\frac {1}{2}\E \[(X-t)^2_+\] =\left\{\begin{array}{ll}
     \dfrac {t^2}{2} +\dfrac {8}{4e^2-1}, & t\le -\dfrac {4}{4e^2-1},\\[8pt]
     \dfrac {1}{8e^2} (t-4)^2, & -\dfrac {4}{4e^2-1}<t<4,\\[3pt]
     0, & t\ge 4; \end{array}\right.
\end{align*}
and
\begin{align*}
  \Fbar_Y^{[2]}(t) & =\E [(Y-t)_+] =\left\{\begin{array}{ll}
     \dfrac {1}{2} e^t -t, & t\le 0,\\[8pt]
     \dfrac {1}{2} e^{-t}, & t>0; \end{array}\right.  \\[5pt]
  \Fbar_Y^{[3]}(t) & =\frac {1}{2}\E \[(Y-t)^2_+\] =\left\{\begin{array}{ll}
     1+\dfrac {t^2}{2} -\dfrac {1}{2} e^t, & t\le 0,\\[5pt]
     \dfrac {1}{2} e^{-t}, & t>0. \end{array}\right.
\end{align*}
One can easily verify that $\E X=\E Y=0$, $\E [X^2]<\E [Y^2]$ and $\Fbar_X^{[3]}(t)\le \Fbar_Y^{[3]}(t)$ for all $t\in\R$. Thus, $X\le_{3\hbox{-}{\rm wcx}} Y$. Assume that there exists $Z\in L^2$ such that $X\le_{3\hbox{-}{\rm cx}} Z\le_{2\hbox{-} {\rm cx}} Y$. Then $\Fbar^{[3]}_X(t)\le \Fbar_Z^{[3]}(t)$ and $\Fbar^{[2]}_Z(t)\le \Fbar_Y^{[2]}(t)$ for all $t\in\R$. Moreover, $\E Z=0$ and $\E [Z^2]=\E [X^2]$. Suppose there exists $t_0>2$ such that $\Fbar^{[2]}_Z(t_0) < \Fbar_Y^{[2]}(t_0)$. By the continuity of $\Fbar_X^{[2]}, \Fbar_Y^{[2]}$ and $\Fbar^{[2]}_Z$, we have
$$
   \Fbar^{[3]}_Z(2)=\int^\oo_2 \Fbar^{[2]}_Z(x)\d x < \int^\oo_2\Fbar^{[2]}_Y(x)\d x =\Fbar^{[3]}_Y(2)=\Fbar^{[3]}_X(2),
$$
which leads to a contradiction. Thus, by the continuity of $\Fbar^{[2]}_Y$ and $\Fbar^{[2]}_Z$, we have 
$$
   \Fbar^{[2]}_Z(t)=\Fbar^{[2]}_Y(t)\ \hbox{for all}\ t\ge 2.
$$
Consequently, by the convexity of $\Fbar^{[2]}_Z$, we have
\begin{align*}
  \Fbar^{[2]}_Z(t) &\ge D^+\Fbar^{[2]}_Z(2) (t-2) + \Fbar^{[2]}_Z(2) 
    = D^+\Fbar^{[2]}_Y(2) (t-2) + \Fbar^{[2]}_Y(2) \\
   & =  -\frac {1}{2e^2} (t-2)+\frac {1}{2e^2} = \frac {1}{2e^2} (3-t).
\end{align*}
Additionally, since $\Fbar_Z^{[2]}(t) =\E [(Z-t)_+]\ge \E [Z-t]=-t$, we obtain
$$
   \Fbar_Z^{[2]}(t) \ge \max \left\{-t,\ \frac {1}{2e^2} (3-t)\right \},\quad t<2.
$$
So, by Proposition \ref{pr-260202} with $n=3$, we have
\begin{align*}
  \frac {1}{2} \big (\E [Z^2] -\E [X^2]\big ) 
  &= \int^2_{-\oo} \(\Fbar_Z^{[2]}(t) -\Fbar_X^{[2]}(t)\) \d t + \Fbar_Z^{[3]}(2) -\Fbar_X^{[3]}(2) \\ 
  & \ge \int^2_{-\oo} \(\Fbar_Z^{[2]}(t) -\Fbar_X^{[2]}(t)\) \d t\\ 
  & \ge \int^2_{-4/(4e^2-1)} \[\frac {1}{2e^2} (3-t) +\frac {1}{4e^2} (t-4)\]\d t \\
  &= \frac {1}{2e^2} \(\frac {4e^2+1}{4e^2-1}\)^2 >0,  
\end{align*}
which contradicts $\E [Z^2]=\E [X^2]$. Therefore, given $X\le_{\rm 3\hbox{-}wcx} Y$, there does not exist any random variable $Z\in L^2$ such that $X\le_{\rm 3\hbox{-}cx} Z\le_{\rm 2\hbox{-}cx} Y$.    }
\end{example}

\section{Discussion}

In this paper, we revisit the proofs for the separation theorem of the increasing convex order, and establish the separation theorem for the weak $3$-convex order over all real-valued random variables via a constructive approach based on higher-order integrated survival functions. A counterexample is presented to prove that the separation of weak 3-convex order is only one-way valid, that is, a weak $3$-convex order shift can not be decomposed into a $3$-convex order followed by a $2$-convex order. 

\cite{DLM14} attempted to generalize the separation theorem to weak $n$-convex orders with $n\ge 4$, but the corresponding proof remains incomplete. In view of the consistent structural features across low-order and high-order weak convex orders, we have the following conjecture:

\begin{conjecture}
Given two random variables $X,Y\in L^{n-1}$ for $n>2$, $X\le_{n\hbox{-}{\rm wcx}} Y$ if and only if there exists a random variable $Z\in L^{n-1}$ such that $X\le_{(n-1)\hbox{-}{\rm cx}} Z\le_{n\hbox{-}{\rm cx}} Y$.
\end{conjecture}


Similar to the notion of $n$-th degree $m$-mean preserving stochastic dominance \citep{Liu14,WW25}, we can define $n$-th degree $m$-preserving convex order, denoted as $(n,m)$-CX order, where $m$ and $n$ are integers such that $0\le m<n$. For $X, Y\in L^{n-1}$, $X$ is said to be smaller than $Y$ in the sense of $(n,m)$-CX order, written $X\le_{(n,m)\hbox{-}{\rm cx}} Y$, if $\Fbar^{[n]}_X (x)  \le  \Fbar^{[n]}_Y (x)$ for all $x\in\R$, and $\E [X^k]=\E [Y^k]$ for $k=1, \ldots, m$. The weak $n$-convex order and $n$-convex order are particular cases of $(n,m)$-CX order with $m=n-2$ and $m=n-1$, respectively.

\begin{conjecture}
Let $X,Y\in L^{n-1}$ and $n>m\ge 1$.
\begin{itemize}
  \item[{\rm (1)}] $X\le_{(n,m)\hbox{-}{\rm cx}} Y$ if and only if there exist andom variable $Z_1,\ldots, Z_{n-m-1}\in L^{n-1}$ such that 
    $$
        X\le_{(m+1,m)\hbox{-}{\rm cx}} Z_1\le_{(m+2,m+1)\hbox{-}{\rm cx}} \cdots \le_{(n-1,n-2)\hbox{-}{\rm cx}} Z_{n-m-1}  \le_{(n,n-1)\hbox{-}{\rm cx}}  Y.
    $$  
  \item[{\rm (2)}] For $n>k>m$, $X\le_{(n,m)\hbox{-}{\rm cx}} Y$ if and only if there exist andom variable $Z\in L^{n-1}$ such that 
    $$
        X\le_{(k,m)\hbox{-}{\rm cx}} Z \le_{(n,k)\hbox{-}{\rm cx}} Y.
    $$  
\end{itemize}
\end{conjecture}

\section*{Funding}

Z. Zou gratefully acknowledges financial support from National Natural Science Foundation of China (No. 12401625), and the Fundamental Research Funds for the Central Universities (No. WK2040000108). T. Hu would like to acknowledge financial support from National Natural Science Foundation of China (No. 72332007, 12371476).

\section*{Disclosure statement}

No potential conflict of interest was reported by the authors.

\end{document}